\input ./style/arxiv-vmsta.cfg
\documentclass[numbers,compress,v1.0.1]{vmsta}

\volume{3}
\issue{2}
\pubyear{2016}
\firstpage{191}
\lastpage{208}
\doi{10.15559/16-VMSTA58}

\startlocaldefs

\allowdisplaybreaks

\newtheorem{thm}{Theorem}[section]
\newtheorem{lem}{Lemma}[section]

\theoremstyle{definition}
\newtheorem{zau}{Remark}[section]

\DeclareMathOperator{\pr}{\mathsf P}
\DeclareMathOperator{\M}{\mathsf E}

\newcommand{\R}{\mathbb{R}}

\endlocaldefs

\begin{document}
\begin{frontmatter}

\title{Asymptotic behavior of homogeneous additive functionals of the
solutions of It\^{o}
stochastic differential equations with nonregular dependence on parameter}

%
%
%

\author{\inits{G.}\fnm{Grigorij}\snm{Kulinich}}\email{zag\_mat@univ.kiev.ua}
\address{ Taras Shevchenko National University of Kyiv, 64/13,\break
Volodymyrska Street, 01601, Kyiv, Ukraine}

\author{\inits{S.}\fnm{Svitlana}\snm{Kushnirenko}\corref{cor1}}\email{bksv@univ.kiev.ua}
\cortext[cor1]{Corresponding author.}

\author{\inits{Yu.}\fnm{Yuliia}\snm{Mishura}}\email{myus@univ.kiev.ua}

\markboth{G. Kulinich et al.}{Asymptotic behavior of homogeneous additive functionals}

\begin{abstract}
We study the asymptotic behavior of mixed functionals of the form
$I_T(t)=\break F_T(\xi_T(t))+\int_{0}^{t} g_T(\xi_T(s))\,d\xi_T(s)$, $
t\ge0$, as $T\to\infty$. Here $\xi_T(t)$ is a strong solution of the
stochastic differential equation $d\xi_T (t)=a_T(\xi_T(t))\,
dt+dW_T(t)$, $T>0$ is a parameter, $a_T= a_T(x)$ are measurable
functions such that $\left|a_T(x)\right|\leq C_T$ for all $x\in\R$,
$W_T(t)$ are standard Wiener processes, $ F_T= F_T(x)$, $x\in\R$, are
continuous functions, $ g_T= g_T(x)$, $x\in\R$, are locally bounded
functions, and everything is real-valued. The explicit form of the
limiting processes for $I_T(t)$ is established under very nonregular
dependence of $ g_T $ and $a_T $ on the parameter $T$.
\end{abstract}

%

\received{28 May 2016}
\revised{17 June 2016}
\accepted{17 June 2016}
\publishedonline{4 July 2016}

\end{frontmatter}

\section{Introduction}

Consider the It\^{o} stochastic differential equation
\begin{equation}\label{o1}
d\xi_T (t)=a_T\bigl(\xi_T(t)\bigr)\,dt+dW_T(t),\quad t\ge0,\;\xi_T (0)=x_0,
\end{equation}
where $T>0$ is a parameter, $a_T (x)$, $x\in\R$, are real-valued
measurable functions such that for some constants $L_T>0$ and for all
$x\in\mathbb{R}$ $\left|a_T(x)\right|\leq L_T$, and $W_T=\{W_T(t),
t\geq0\}$, $T>0$, is a family of standard Wiener processes defined on
a complete probability space $(\varOmega,\Im, \pr)$.

It is known from Theorem 4 in \cite{paper1} that, for any $T>0$ and
$x_0\in\R$, equation (\ref{o1}) possesses a unique strong pathwise
solution $\xi_T=\{\xi_T(t), t\geq0\}$, and this solution is a
homogeneous strong Markov process.


We suppose that the drift coefficient $a_T (x)$ in equation
(\ref{o1}) can have a very nonregular dependence on the
parameter. For example, the drift coefficient can be of
``$\delta$''-type sequence at some points $x_k$ as $T\to
\infty$, or it can be equal to $\sqrt{T}\sin
((x-x_k)\sqrt{T})$, or it can have degeneracies of
some other types. Such a nonregular dependence of the
coefficients in equation (\ref{o1}) first appeared in
\cite{paper3} and \cite{paper4}, where the limit behavior of the
normalized unstable solution of It\^{o} stochastic differential
equation as $t\to\infty$ was investigated. In those papers, a
special dependence of the coefficients $a_T
(x)=\sqrt{T}a(x\sqrt{T})$ on the parameter $T$ was considered in
the case where $a(x)$ is an absolutely integrable function on $\mathbb
{R}$. Assume that this is the case and let $\int_{\mathbb{R}}
a(x)\,dx=\lambda$. The sufficiency of the condition $\lambda=0$ for
the asymptotic equivalence of distributions $\xi_T$ and $W_{T}(t)$ is
established in \cite{paper3}, and the necessity of this
condition is proved in \cite{paper45}. If $\lambda\neq0$,
then we can deduce from \cite{paper4} that the
distributions of the solution $\xi_T$ of equation (\ref{o1}) weakly
converge as $T\to\infty$ to the corresponding distributions of
the Markov process $\hat{\xi}(t)=l\left(\zeta(t)\right)$, where
$l(x)=c_1x$ for $x>0$ and $l(x)=c_2x$ for $x\leq0$; $\zeta(t)$ is a
strong solution of the It\^{o} equation $d\zeta(t)=\bar
\sigma(\zeta(t))\,d W(t)$, where $\bar\sigma(x)=\sigma_1$ for
$x>0$ and $\bar\sigma(x)=\sigma_2$ for $x<0$, and
$\int_{0}^{t}P\{|\zeta(s)|=0\}\,ds=0$. The explicit form of
the transition density of the process $\hat{\xi}(t)$ is obtained.
Moreover, in \cite{paper5}, it is proved that
\[
\hat{\xi}(t)=x_0+\beta(t)+W(t),
\]
where $\beta(t)$ is a certain
functional of $\zeta(t)$, and the necessity of the condition $\lambda
\neq0$ is established for the weak convergence as $T\to\infty$ of
the solution $\xi_T$ of equation (\ref{o1}) to the process
$\hat{\xi}(t)$.

Furthermore, in \cite{paper3} and \cite{paper4}, a probabilistic method
to study the ``awkward'' term $\sqrt{T}\int_{0}^{t}a(\xi
_T(s)\sqrt{T})\,ds$ in equation (\ref{o1}) is developed. This method
uses a representation of this ``awkward'' term through a family of
continuous functions $\varPhi_T (x)$ of $\xi_T(t)$ and a family of
martingales $\int_{0}^{t}\varPhi'_T(\xi_T(s))\,dW_T(s)$, with the
further application of the It\^{o} formula. After the mentioned
transformations, according to this method, we can apply Skorokhod's
convergent subsequence principle for $\xi_{T_n}(t)$ and $W_{T_n}(t)$
(see \cite{book6}, Chapter I, \S6) in order to pass to the limit in
the resulting representation.

Note that this method is also used in the present paper to study the
asymptotic behavior of integral functionals.

It is known from \cite{book2}, \S16, that the asymptotic behavior of
the solution $\xi_T$ of equation (\ref{o1}) is closely related to the
asymptotic behavior of harmonic functions, that is, functions
satisfying the following ordinary differential equation almost
everywhere (a.e.) with respect to the Lebesgue measure:
\[
f_T'(x)a_T(x)+\frac{1}{2}\,f_T''(x)=0.
\]

It is obvious that the functions $f_T(x)$ have the form
\begin{equation}\label{o2}
f_T(x)=c_T^{(1)}\int_{0}^{x}\exp\Bigg\{-2\int
_{0}^{u}a_T(v)\,dv\Bigg\}\,du+c_T^{(2)},
\end{equation}
where $c_T^{(1)}$ and $c_T^{(2)}$ are some families of constants.

The latter functions possess the continuous derivatives
$f_T^{\prime}(x)$, and their second derivatives $f_T^{\prime\prime}(x)$ exist
almost everywhere with respect to the Lebesgue measure and
are locally integrable. Note that $c_T^{(1)}$ are normalizing
constants and $c_T^{(2)}$ are centralizing constants in the limit
theorems (see \cite{book7}, \S6). Further, for simplicity, we
assume that in (\ref{o2}), $c_T^{(1)}\equiv1$ and $c_T^{(2)}\equiv
0$.

In this paper, we assume for the coefficient $a_T(x)$ of equation
(\ref{o1}) that there exists a family of functions $G_T(x)$,
$x\in\R$, with continuous derivatives $G_T^{\prime}(x)$ and locally
integrable second derivatives $G_T^{\prime\prime}(x)$ a.e. with respect
to the Lebesgue measure such that, for all $T>0$ and $x\in\R$, the
following inequalities hold:
\begin{align*}
(A_1) \hspace*{27pt}\quad&\biggl( G_T^{\prime}(x)a_T(x)+\frac{1}{2}\,G_T^{\prime\prime}(x)
\biggr)^2+\bigl( G_T^{\prime}(x)\bigr)^2 \leq C \bigl(1+
\bigl(G_T(x)\bigr)^2\bigr),\hspace*{27pt}\\
&\quad\big|G_T(x_0)\big|\leq C.
\end{align*}

Suppose additionally that the functions $G_T(x)$, $x\in\R$, introduced
by condition $(A_1)$ satisfy the following assumptions:
\begin{itemize}
\item[$(i)$] There exist constants $C>0$ and $\alpha>0$ such that
$|G_T(x)|\geq C |x|^{\alpha}$.
\item[$(ii)$] There exist a bounded function $\psi\left(|x|\right)$
and a constant $m\geq0$ such that\break $\psi\left(|x|\right)\to0$ as
$|x|\to0$ and, for all $x\in\R$ and $T>0$ and for any measurable
bounded set $B$, the following inequality holds:
\begin{align*}
(A_2) \hspace*{15pt}\quad\quad\int_{0}^{x} f_T'(u) \Biggl( \,\int
_{0}^{u}\frac{\chi_B\left(G_T(v)\right)}{f'_T(v)}\,dv\Biggr)\,du
\leq\psi\bigl(\lambda(B)\bigr) \left[1+ |x|^m\right],\hspace*{15pt}
\end{align*}
where $\chi_B(v)$ is the indicator function of a set $B$, $\lambda
(B)$ is the Lebesgue measure of $B$, and $f_T'(x)$ is the derivative of
the function $f_T(x)$ defined by equality (\ref{o2}).
\end{itemize}

Let $\left\{G_T \right\}$ be the class of the functions $G_T(x)$, $x\in
\R$, satisfying conditions $(A_1)$ and (i)--(ii). The class of
equations of the form (\ref{o1}) whose coefficients $a_T(x)$ admit
$G_T(x)$, $x\in\R$, from the class $\left\{G_T \right\}$ will be
denoted by $K\left(G_T \right)$. It is easy to understand that class
$K\left(G_T \right)$ does not depend on the constants $c_T^{(1)}$ and
$c_T^{(2)}$ in representation (\ref{o2}).

It is clear that if there exist constants $\delta> 0$ and $C>0$ such
that $0<\delta\leq f_T'(x)\leq C$ for all $x\in\R$, $T>0$, then
the corresponding equations (\ref{o1}) belong to the class $K\left(G_T
\right)$ for $G_T(x)=f_T(x)$. We denote this subclass as $K_1$.
Note that the class $K\left(G_T \right)$ contains in particular
the equations for which, at some points $x_k$, we have the
convergence $f_T'(x_k)\to\infty$ or the convergence $f_T'(x_k)\to
0$ as $T\to\infty$. For example, consider equation (\ref{o1})
with $a_T(x)=\frac{c_0Tx}{1+x^2T}$. It is easy to obtain that
$f_T'(x)=\frac{1}{\left(1+x^2T\right)^{c_0}}$, and if
$c_0>-\frac{1}{2}$, then such equations belong to the class $K\left(G_T
\right)$ with $G_T(x)=x^2$ (here, at points $x\neq0$, we have
$f_T'(x)\to0$ for $c_0>0$, $f_T'(x)\to\infty$ for
$-\frac{1}{2}<c_0<0$, and $f_T'(x)\equiv1$ for $c_0=0$).\vadjust{\eject}

For the class of equations $K\left(G_T \right)$, we study the
asymptotic behavior as $T\to\infty$ of the distributions of the
following functionals:
\begin{align*}
&\beta_T^{(1)} (t)=\int_{0}^{t} g_T\bigl(\xi_T(s)\bigr)\,ds,\qquad \beta
_T^{(2)} (t)=\int_{0}^{t} g_T\bigl(\xi_T(s)\bigr)\,dW_T(s),\\
&I_T(t)=F_T\bigl(\xi_T(t)\bigr)+\int_{0}^{t} g_T\bigl(\xi_T(s)\bigr)\,dW_T(s),\qquad
\beta_T (t)=\int_{0}^{t} g_T\bigl(\xi_T(s)\bigr)\,d\xi_T(s),
\end{align*}
where the processes $\xi_T(t)$, $W_T(t)$ are related via \xch{equation~(\ref
{o1})}{Eq.~(\ref
{o1})}, $ g_T (x) $ is a family of measurable locally bounded
real-valued functions, and $ F_T (x)$ is a family of continuous
real-valued functions.

This paper is a continuation of \cite{paper77,paper88,paper8}. Note
that the behavior of the distributions of functionals $\beta_T^{(1)}
(t)$, $\beta_T^{(2)} (t)$ for the solutions $\xi_T$ of equations (\ref
{o1}) from the class $K_1$ is studied in \cite{paper55} and \cite
{paper56}. The case where $W_T(t)$ is replaced with $\eta_T(t)$, where
$\eta_T(t)$ is a family of continuous martingales with the
characteristics $\langle{\eta_T}\rangle(t)\to t$ as $T\to\infty$, was
studied in \cite{paper57}. Paper \cite{paper58} was devoted to a
discrete analogue of the results from \cite{paper57}. A similar problem
for the functionals $I_T(t)$ in the case of equation (\ref{o1}) with
$a_T (x)\equiv0$ was considered in \cite{book7} and in \cite{paper11},
for the class $K_1$.
In \cite{paper77,paper88,paper8}, the behavior of the distributions of
the functionals $\beta_T^{(1)} (t)$, $\beta_T^{(2)} (t)$, and $I_T(t)$
with a special dependence of the drift coefficients $a_T (x)=\sqrt
{T}a(x\sqrt{T})$ on the parameter $T$ is considered, mainly in the case
where $\left|xa(x)\right|\leq C$ for all $x\in\mathbb{R}$. The behavior
of the distributions of functionals $\beta_T^{(1)} (t)$ was studied in
\cite{paper77}, $\beta_T^{(2)} (t)$ was studied in \cite{paper88}, and
$I_T(t)$ was investigated in \cite{paper8}. A more detailed review of
the known results in this area is presented in \cite{paper77,paper88,paper8}. Note that the functionals $\beta_T^{(1)} (t)$, $\beta_T^{(2)}
(t)$, and $\beta_T (t)$ are particular cases of the functional $I_T(t)$
(see \cite{paper8}, Lemma 4.1).

\begin{zau}\label{z3}

In this paper, we often apply the It\^{o} formula to the process $\varPhi
(\xi_T(t))$, where $\xi_T(t)$ is a solution of equation (\ref{o1}), the
derivative $\varPhi'(x)$ of the function
$\varPhi(x)$ is assumed to be continuous, and the second derivative $\varPhi
''(x)$ is assumed to exist a.e.\  with respect to the Lebesgue measure
and to be locally integrable. Then it follows from \cite{paper9} that
with probability one, for all $t\geq0$, the following equality holds:
\begin{align*}
\varPhi\bigl(\xi_T(t)\bigr)-\varPhi(x_0)&= \int_{0}^{t}\biggl(\varPhi
'\bigl(\xi_T(s)\bigr)a_T\bigl(\xi_T(s)\bigr)+\frac{1}{2}\,\varPhi''\bigl(\xi_T(s)\bigr)\biggr)\,ds\\
&\quad+\int
_{0}^{t}\varPhi'\bigl(\xi_T(s)\bigr)\,dW_T(s).
\end{align*}
\end{zau}

\begin{zau}\label{zZ3} Let $\xi_T$ be a solution of equation (\ref
{o1}), and $G_T(x)$ be a family of functions satisfying condition $(A_1)$.
Theorem 1 from \cite{paper10} implies that the family of
the processes $\{\zeta_T(t)=G_T(\xi_T(t)), t\geq0\}$ is weakly
compact. The proof of this result is based on the equality
\begin{equation}\label{o4}
\zeta_T(t)=G_T(x_0)+\int_{0}^{t}\biggl( G_T^{\prime}\bigl(\xi_T(s)\bigr)a_T\bigl(\xi
_T(s)\bigr)+\frac{1}{2}\,G_T^{\prime\prime}\bigl(\xi_T(s)\bigr) \biggr)\,ds+\eta_T(t),
\end{equation}
where
\[
\eta_T(t)=\int_{0}^{t}G_T^{\prime}\bigl(\xi_T(s)\bigr)\,dW_T(s),\qquad\zeta
_T (t)=G_T\bigl(\xi_T(t)\bigr).
\]
In turn, the latter equality follows from Remark \ref{z3}.
In addition, it is established in the proof of Theorem 1 from \cite
{paper10} that for any constants $L>0$ and $\varepsilon>0$,
\begin{align}\label{o5}
\begin{split}
\mathop{\lim}\limits_{N\to\infty}\overline{\mathop{\lim}\limits
_{T \to\infty}} \mathop{\sup}\limits_{0\leq t\leq L} \pr\bigl\{
\big|\lambda_T(t)\big|>N\bigr\}&=0, \\
\mathop{\lim}\limits_{h\to0}\overline{\mathop{\lim}\limits_{T \to
\infty}} \mathop{\sup}\limits_{|t_1-t_2|\leq h;\, t_i\leq L} \pr
\bigl\{\big|\lambda_T(t_2)-\lambda_T(t_1)\big|>\varepsilon\bigr\}&=0,
\end{split}
\end{align}
and that, for any $k>1$ and for certain constants $C_k$ and $C$,
\begin{equation}\label{o6}
\M\mathop{\sup}\limits_{0\leq t\leq L} \big|\lambda_T(t)\big|^k\leq C_k,\qquad
\M\big|\lambda_T(t_2)-\lambda_T(t_1)\big|^4\leq C|t_2-t_1|^2,
\end{equation}
where $\lambda=\eta$ or $\lambda=\zeta$ (see \cite{book2}, \S6,
Theorem 4).
\end{zau}

\begin{zau}\label{z1}
Here and throughout the paper, the weak convergence of the processes
means the weak convergence in the uniform topology of the space of
continuous functions $C[0,L]$ for any $L>0$. The processes that have
continuous trajectories with probability 1 will be simply called continuous.
\end{zau}

The paper is organized as follows. Section~\ref{section2} contains the statements of
the main results. In Section~\ref{section3}, they are proved. Auxiliary results are
collected in Section~\ref{section4}.

\section{Statement of the main results}\label{section2}

In what follows, we denote by $C,\,L,\,N,\,C_N$ any constants that do
not depend on $T$ and $x$. Assume that, for certain locally bounded
functions $ q_T (x) $ and any constant $N>0$, the following condition holds:
\[
(A_3)\hspace*{40pt}\quad\quad\quad\quad\lim\limits_{T\to\infty}\mathop{\sup}\limits
_{|x|\leq N} f_T'(x) \Biggl| \,\int_{0}^{x}\frac
{q_T(v)}{f'_T(v)}\,dv\Biggr|=0,\hspace*{80pt}
\]
where $f_T'(x)$ is the derivative of the function $f_T(x)$ defined by
Eq.~(\ref{o2}).

\begin{thm}\label{th2}
Let $\xi_T$ be a solution of Eq.~\eqref{o1} from the class
$K\left(G_T \right)$ and $G_T(x_0)\to y_0$ as $T\to\infty$.
Assume that there exist measurable locally bounded functions
$a_0(x)$ and $\sigma_0(x)$ such that:
\begin{itemize}
\item[\rm1.] the functions
\[
q_T^{(1)}(x)=G_T^{\prime}(x)a_T(x)+\frac{1}{2}\,G_T^{\prime\prime}(x)- a_0\bigl(G_T(x)\bigr)
\]
and
\[
q_T^{(2)}(x)=\bigl(G_T^{\prime}(x)\bigr)^2- \sigma_0^2\bigl(G_T(x)\bigr)
\]

satisfy assumption $(A_3)$;

\item[\rm2.] the It\^{o} equation
\begin{equation}\label{o3}
\zeta(t)=y_0+\int_{0}^{t}a_0\bigl(\zeta(s)\bigr)\,ds+\int
_{0}^{t}\sigma_0\bigl(\zeta(s)\bigr)\,d\hat{W}(s)
\end{equation}
has a unique weak solution.
\end{itemize}
Then the stochastic process $\zeta_T(t)=G_T(\xi_T(t))$ weakly
converges, as $T\to\infty$, to the solution $\zeta(t)$ of Eq.~\eqref{o3}.\vadjust{\eject}
\end{thm}

\begin{thm}\label{th3}
Let $\xi_T$ be a solution of Eq.~\eqref{o1} from the class $K\left(G_T
\right)$, and let the assumptions of Theorem \emph{\ref{th2}} hold. Assume
that, for measurable locally bounded functions $g_T(x)$, there exists a
measurable locally bounded function $g_0(x)$ such that the function
\[
q_T(x)=g_T(x)- g_0\bigr(G_T(x)\bigl)
\]
satisfies assumption $(A_3)$. Then the stochastic process $ \beta
_T^{(1)} (t)=\int_{0}^{t} g_T(\xi_T(s))\,ds$ weakly converges,
as $T\to\infty$, to the process
\[
\beta^{(1)} (t)=\int_{0}^{t} g_0\bigl(\zeta(s)\bigr)\,ds,
\]
\noindent where $\zeta(t)$ is a solution of Eq.~\eqref{o3}.
\end{thm}

\begin{thm}\label{th4}
Let $\xi_T$ be a solution of Eq.~\eqref{o1} from the class $K\left(G_T
\right)$, and let the assumptions of Theorem \emph{\ref{th2}} hold.
Assume that, for measurable locally bounded functions $g_T(x)$, there
exists a measurable locally bounded function $g_0(x)$ such that
\[
(A_4)\hspace*{15pt}\quad\quad\quad\lim\limits_{T\to\infty}\mathop{\sup}\limits
_{|x|\leq N} \Biggl| f_T'(x) \,\int_{0}^{x}\frac
{g_T(v)}{f'_T(v)}\,dv-g_0\bigl(G_T(x)\bigr)G_T^{\prime}(x)\Biggr|=0\hspace*{39pt}
\]
for all $N>0$. Then the stochastic process $ \beta_T^{(1)} (t)=\int
_{0}^{t} g_T(\xi_T(s))\,ds$ weakly converges, as $T\to\infty$,
to the process
\[
\tilde{\beta}^{(1)} (t)=2\Biggl(\int_{y_0}^{\zeta(t)} g_0(x)\,
dx-\int_{0}^{t} g_0\bigl(\zeta(s)\bigr)\,\sigma_0\bigl(\zeta(s)\bigr)\,d\hat
{W}(s)\Biggr),
\]
where $\zeta(t)$ and the Wiener process $\hat{W}(t)$ are related via
Eq.~\eqref{o3}.
\end{thm}

\begin{thm}\label{th5}
Let $\xi_T$ be a solution of equation \emph{(\ref{o1})} from the class $K\left
(G_T \right)$, and let the assumptions of Theorem \emph{\ref{th2}} hold.
Assume that, for measurable locally bounded functions $g_T(x)$, there
exists a measurable locally bounded function $g_0(x)$ such that
the function
\[
q_T(x)=\bigl(g_T(x)- g_0\bigl(G_T(x)\bigr)G_T^{\prime}(x)\bigr)^2
\]
satisfies assumption $(A_3)$. Then the stochastic process
$\beta_T^{(2)} (t)\,{=}\int_{0}^{t} g_T(\xi_T(s))\,dW_T(s)$,
where $\xi_T(t)$ and $W_T(t)$ are related via Eq.~\eqref{o1}, weakly
converges, as $T\to\infty$, to the process
\[
{\beta}^{(2)} (t)=\int_{0}^{t} g_0\bigl(\zeta(s)\bigr)\,d\zeta(s)-\int
_{0}^{t} g_0\bigl(\zeta(s)\bigr)\,a_0\bigl(\zeta(s)\bigr)\,ds,
\]
where $\zeta(t)$ is a solution of Eq.~\eqref{o3}.
\end{thm}

\begin{thm}\label{th6}
Let $\xi_T$ be a solution of Eq.~\eqref{o1} from the class $K\left(G_T
\right)$, and let the assumptions of Theorem \emph{\ref{th2}} hold.
Assume that, for continuous functions $F_T(x)$ and locally bounded
measurable functions $g_T(x)$, there exist a continuous function
$F_0(x)$ and locally bounded measurable function $g_0(x)$ such that,
for all~$N>0$,
\[
\lim\limits_{T\to\infty}\mathop{\sup}\limits_{|x|\leq N} \big|
F_T(x)-F_0\bigl(G_T(x)\bigr) \big|=0,
\]
and let the functions $g_T(x)$ and $g_0(x)$ satisfy the assumptions of
Theorem \emph{\ref{th5}}. Then the stochastic process\vadjust{\eject}
\[
I_T(t)=F_T\bigl(\xi_T(t)\bigr)+\int_{0}^{t} g_T\bigl(\xi_T(s)\bigr)\,dW_T(s),
\]
where $\xi_T(t)$ and $W_T(t)$ are related via Eq.~\eqref{o1}, weakly
converges, as $T\to\infty$, to the process
\[
I_0(t)=F_0\bigl(\zeta(t)\bigr)+\int_{0}^{t} g_0\bigl(\zeta(s)\bigr)\,\sigma_0\bigl(\zeta
(s)\bigr)\,d\hat{W}(s),
\]
where $\zeta(t)$ and the Wiener process $\hat{W}(t)$ are related via
Eq.~\eqref{o3}.
\end{thm}

The next theorem principally follows from \cite{paper11}; however, we
provide its proof for the reader's convenience and completeness of the results.

\begin{thm}\label{th7}
Let $\xi_T$ be a solution of Eq.~\eqref{o1} from the class $K\left(G_T
\right)$ for $G_T(x)=f_T(x)$, and let $0<\delta\leq f_T'(x)\leq C$ and
$f_T(x_0)\to y_0$ as $T\to\infty$. Also, let $\zeta_T(t)=f_T(\xi
_T(t))$,
\[
I_T(t)=F_T\bigl(\xi_T(t)\bigr)+\int_{0}^{t} g_T\bigl(\xi_T(s)\bigr)\,dW_T(s),
\]
$F_T(x)$ be continuous functions, $g_T(x)$ be locally square-integrable
functions, and the processes $\xi_T(t)$ and $W_T(t)$ be related via
Eq.~\eqref{o1}.

The two-dimensional process $\left(\zeta_T (t),I_T(t)\right)$ weakly
converges, as $T\to\infty$, to the process $\left(\zeta(t),I(t)\right
)$, where $I(t)=F_0(\zeta(t))+\int_{0}^{t} g_0(\zeta(s))\,d\zeta
(s)$, and $\zeta(t)$ is a~weak solution of the It\^{o} equation $\zeta
(t)=y_0+\int_{0}^{t}\sigma_0(\zeta(s))\,d{W}(s)$, if and only if
there exist constants $c_T^{(1)}$ and $c_T^{(2)}$ in \emph{(\ref{o2})} such
that, as $T\to\infty$:
\begin{itemize}

\item[\rm1.] for all $x$,
\[
\int_{0}^{\varphi_T (x)}\frac{[f_T'(v)]^2- \sigma
_0^2(f_T(v))}{f_T'(v)}\,dv \to0,
\]
where $\varphi_T (x)$ is the inverse function of the function $f_T(x)$;

\item[\rm2.] for all $N>0$,
\[
\mathop{\sup}\limits_{|x|\leq N} \big|F_T(x)+f_T(x)-F_0
\bigl(f_T(x)\bigr)\big|\to0
\]
and
\[
\int_{-N}^{N}\frac{|g_T(x)-f_T'(x)[1+g_0
(f_T(x))]|^2}{f_T'(x)}\,dx \to0.
\]
\end{itemize}
\end{thm}

\section{Proof of the main results}\label{section3}

Proof of Theorem \ref{th2}. Rewrite Eq.~(\ref{o4}) as
\begin{equation}\label{o7}
\zeta_T(t)=G_T(x_0)+\int_{0}^{t}a_0\bigl(\zeta_T(s)\bigr)\,ds+\alpha
^{(1)}_T(t)+\eta_T(t),
\end{equation}
where
\[
\alpha^{(1)}_T(t)=\int_{0}^{t}q^{(1)}_T\bigl(\xi_T(s)\bigr)\,ds,\qquad
q^{(1)}_T(x)= G_T^{\prime}(x)a_T(x)+\frac{1}{2}\,G_T^{\prime\prime}(x)-a_0
\bigl(G_T(x)\bigr).\vadjust{\eject}
\]
The functions $q^{(1)}_T(x)$ satisfy the conditions of Lemma \ref{lm2}.
Thus, for any $L>0$,
\begin{equation}\label{o8}
\sup_{0\leq t\leq L}\big| \alpha^{(1)}_T(t)\big|\stackrel{\pr}{\to} 0
\end{equation}
as $T\to\infty$. It is clear that $\eta_T (t)$ is a family of
continuous martingales with quadratic characteristics
\begin{equation}\label{o9}
\langle\eta_T\rangle(t)=\int_{0}^{t}\bigl(G_T^{\prime}\bigl(\xi
_T(s)\bigr)\bigr)^2\,ds=\int_{0}^{t}\sigma_0^2\bigl(\zeta_T(s)\bigr)\,ds+\alpha
^{(2)}_T(t),
\end{equation}
where
\[
\alpha^{(2)}_T(t)=\int_{0}^{t}q^{(2)}_T\bigl(\xi_T(s)\bigr)\,ds,\qquad
q^{(2)}_T(x)= \bigl(G_T^{\prime}(x)\bigr)^2-\sigma_0^2\bigl(G_T(x)
\bigr).
\]

The functions $q^{(2)}_T(x)$ satisfy the conditions of Lemma \ref{lm2}.
Thus, for any $L>0$,
\begin{equation}\label{o10}
\sup_{0\leq t\leq L}\big| \alpha^{(2)}_T(t)\big|\stackrel{\pr}{\to} 0
\end{equation}
as $T\to\infty$.

We have that relations (\ref{o5}) and (\ref{o6}) hold for the
processes $\zeta_T(t)$ and $\eta_T(t)$, and, according to (\ref{o8})
and (\ref{o10}), these relations hold for the processes $\alpha
^{(k)}_T(t)$, $k=1,2$, as well.
This means that we can apply Skorokhod's convergent subsequence
principle (see \cite{book6}, Chapter I, \S6) for the process $
(\zeta_T(t), \eta_T(t),\alpha^{(1)}_T(t), \alpha^{(2)}_T(t))$:
given an arbitrary sequence $T'_n\to\infty$, we can choose a subsequence
$T_n\to\infty$, a~probability space $(\tilde\varOmega,\tilde\Im, \tilde
\pr)$, and a stochastic process $(\tilde{\zeta}_{T_n}(t), \tilde
{\eta}_{T_n}(t),\tilde{\alpha}^{(1)}_{T_n}(t), \tilde{\alpha
}^{(2)}_{T_n}(t))$ defined on this space such that its
finite-dimensional distributions coincide with those of the process
$(\zeta_{T_n}(t), \eta_{T_n}(t),\alpha^{(1)}_{T_n}(t), \alpha
^{(2)}_{T_n}(t))$
and, moreover,
\begin{align*}
\tilde{\zeta}_{T_n}(t)\stackrel{\tilde{\pr}}{\to}\tilde{\zeta}(t),\qquad
\tilde{\eta}_{T_n}(t)\stackrel{\tilde{\pr}}{\to}\tilde{\eta}(t),\qquad
\tilde{\alpha}^{(1)}_{T_n}(t)\stackrel{\tilde{\pr}}{\to}\tilde{\alpha
}^{(1)}(t),\qquad \tilde{\alpha}^{(2)}_{T_n}(t)\stackrel{\tilde{\pr}}{\to
}\tilde{\alpha}^{(2)}(t)
\end{align*}
for all $0\leq t\leq L$, where $\tilde{\zeta}(t)$, $\tilde{\eta}(t)$,
$\tilde{\alpha}^{(1)}(t)$, $\tilde{\alpha}^{(2)}(t)$ are some
stochastic processes.

Evidently, relations (\ref{o8}) and (\ref{o10}) imply that $\tilde
{\alpha}^{(k)}(t)\equiv0$, $k=1,2$, a.s. According to (\ref{o6}), the processes
$\tilde{\zeta}(t)$ and $\tilde{\eta}(t)$ are continuous. Moreover,
applying Lemma \ref{lm5} together with Eqs.~(\ref{o7}) and (\ref{o9}),
we obtain that
\begin{align}\label{o11}
\begin{split}
\tilde{\zeta}_{T_n} (t)&=G_{T_n}(x_0)+\int_{0}^{t}a_0\bigl(\tilde{\zeta
}_{T_n}(s)\bigr)\,ds+\tilde{\alpha}^{(1)}_{T_n}(t)+\tilde{\eta}_{T_n}(t), \\
\langle\tilde{\eta}_{T_n}\rangle(t)&=\int_{0}^{t}\sigma
_0^2\bigl(\tilde{\zeta}_{T_n}(s)\bigr)\,ds+\tilde{\alpha}^{(2)}_{T_n}(t),
\end{split}
\end{align}
where $\tilde{\zeta}_{T_n}(t)\stackrel{\tilde{\pr}}{\to}\tilde{\zeta
}(t)$, $\tilde{\eta}_{T_n}(t)\stackrel{\tilde{\pr}}{\to}\tilde{\eta
}(t)$, $\sup_{0\leq t\leq L}| \tilde{\alpha
}^{(k)}_{T_n}(t)|\stackrel{\tilde{\pr}}{\to} 0$, $k=1,2$, as
$T_n\to\infty$. In addition,
it is established in \cite{paper10} that, for any constants $L>0$ and
$\varepsilon>0$,
\[
\mathop{\lim}\limits_{h\to0}\overline{\mathop{\lim}\limits_{T_n
\to\infty}} \tilde{\pr}\Big\{\mathop{\sup}\limits_{|t_1-t_2|\leq
h;\, t_i\leq L} \big|\tilde{\zeta}_{T_n}(t_2)-\tilde{\zeta
}_{T_n}(t_1)\big|>\varepsilon\Big\}=0.
\]

Using the latter convergence and (\ref{o11}), we conclude that, for
any constants $L>0$ and $\varepsilon>0$,\vadjust{\eject}
\[
\mathop{\lim}\limits_{h\to0}\overline{\mathop{\lim}\limits_{T_n
\to\infty}} \tilde{\pr}\Big\{\mathop{\sup}\limits_{|t_1-t_2|\leq
h;\, t_i\leq L} \big|\tilde{\eta}_{T_n}(t_2)-\tilde{\eta
}_{T_n}(t_1)\big|>\varepsilon\Big\}=0.
\]

Therefore, according to the well-known result of Prokhorov
\cite{paper155}, we conclude that
\[
\sup\limits_{0\leq t\leq L}\big| \tilde{\zeta}_{T_n}(t)-\tilde{\zeta
}(t)\big|\stackrel{\tilde{\pr}}{\to} 0 \qquad\mbox{and}\qquad\sup
\limits_{0\leq t\leq L}\big| \tilde{\eta}_{T_n}(t)-\tilde{\eta
}(t)\big|\stackrel{\tilde{\pr}}{\to} 0
\]
as $T_n\to\infty$. According to Lemma \ref{lm3}, we can pass to the
limit in (\ref{o11}) and obtain
\begin{equation}\label{o13}
\tilde{\zeta} (t)=y_0+\int_{0}^{t}a_0\bigl(\tilde{\zeta}(s)\bigr)\,
ds+\tilde{\eta}(t),
\end{equation}
where $\tilde{\eta} (t)$ is a continuous martingale with the quadratic
characteristic
\[
\langle\tilde{\eta}\rangle(t)=\int_{0}^{t}\sigma_0^2\bigl(\tilde
{\zeta}(s)\bigr)\,ds.
\]

Now, it is well known that the latter representation provides the
existence of a Wiener process $\hat{W}(t)$ such that
\begin{equation}\label{o14}
\tilde{\eta} (t)=\int_{0}^{t}\sigma_0\bigl(\tilde{\zeta}(s)\bigr)\,d\hat{W}(s).
\end{equation}
Thus, the process $(\tilde{\zeta} (t),\hat{W}(t))$
satisfies Eq.~(\ref{o3}), and the process $\tilde{\zeta}_{T_n}(t)$
weakly converges, as $T_n\to\infty$, to the process $\tilde{\zeta}(t)$.
Since the subsequence $T_n\to\infty$ is arbitrary and since a solution
of Eq.~(\ref{o3}) is weakly unique, the proof of the Theorem \ref{th2}
is complete.

Proof of Theorem \ref{th3}. It is clear that, for all $t> 0$, with
probability one,
%
\[
\beta_T^{(1)} (t)=\int_{0}^{t} g_0\bigl(\zeta_T(s)\bigr)\,ds+\alpha_T(t),
\]
%
where $\alpha_T(t)=\int_{0}^{t} q_T(\xi_T(s))\,ds$ and
$q_T(x)=g_T(x)-g_0\left(G_T(x)\right)$.

The functions $q_T(x)$ satisfy the conditions of Lemma \ref{lm2}.
Thus, for any $L>0$,
%
\[
\sup_{0\leq t\leq L}\big| \alpha_T(t)\big|\stackrel{\pr}{\to} 0
\]
%
as $T\to\infty$. Similarly to (\ref{o11}), we obtain the equality
\begin{equation}\label{o17}
\tilde{\beta}_{T_n}^{(1)} (t)=\int_{0}^{t} g_0\bigl(\tilde{\zeta
}_{T_n}(s)\bigr)\,ds+\tilde{\alpha}_{T_n}(t),
\end{equation}
where $\tilde{\zeta}_{T_n}(t)\stackrel{\tilde{\pr}}{\to}\tilde{\zeta
}(t)$ and $\sup_{0\leq t\leq L}| \tilde{\alpha
}_{T_n}(t)|\stackrel{\tilde{\pr}}{\to} 0$ as $T_n\to\infty$. The
process $\tilde{\zeta}(t)$ is a solution of Eq.~(\ref{o13}), whereas by
Lemma \ref{lm5} the finite-dimensional distributions of the stochastic
process $\beta_{T_n}^{(1)} (t)$ coincide with those of the process
$\tilde{\beta}_{T_n}^{(1)} (t)$.

Using Lemma \ref{lm3} and Eq.~(\ref{o17}), we conclude that
\[
\sup\limits_{0\leq t\leq L}\Biggl| \tilde{\beta}_{T_n}^{(1)} (t)-\int
_{0}^{t} g_0\bigl(\tilde{\zeta}(s)\bigr)\,ds\Biggr|\stackrel{\tilde{\pr
}}{\to} 0
\]
as $T_n\to\infty$. Thus, the process ${\beta}_{T_n}^{(1)} (t)$ weakly
converges, as $T_n\to\infty$, to the process $ \beta^{(1)} (t)=\int
_{0}^{t} g_0(\zeta(s))\,ds$, where
$\zeta(t)$ is a solution of Eq.~(\ref{o3}). Since the subsequence\vadjust{\eject}
$T_n\to\infty$ is arbitrary and since a solution $\zeta(t)$ of Eq.~(\ref
{o3}) is weakly unique, the proof of Theorem \ref{th3} is complete.

Proof of Theorem \ref{th4}.
Consider the function
\[
\varPhi_T(x)=2\int_{0}^{x}f'_T(u)\Biggl(\int_{0}^{u}\frac
{g_T(v)}{f'_T(v)}\,dv\Biggr)\,du.
\]
Applying the It\^{o} formula to the process $\varPhi_T(\xi_T(t))$, where
$\xi_T(t)$ is a solution of Eq.~(\ref{o1}), we get that
%
\begin{align*}
{\beta}_{T}^{(1)} (t)&=\varPhi_T\bigl(\xi_T(t)\bigr)-\varPhi_T(x_0)- \int
_{0}^{t}\varPhi'_T\bigl(\xi_T(s)\bigr)\,dW_T(s) \\
&=2\int_{x_0}^{\xi_T(t)}g_0\bigl(G_T(u)\bigr)G_T^{\prime}(u)\,du
-2\int_{0}^{t}g_0\bigl({\zeta}_{T}(s)\bigr)G_T^{\prime}\bigl({\xi
}_{T}(s)\bigr)\,dW_T(s) \\
&\quad+2\int_{x_0}^{\xi_T(t)}\hat{q}_T(u)\,du -2\int_{0}^{t}\hat
{q}_T\bigl({\xi}_{T}(s)\bigr)\,dW_T(s)\\
&= 2\int_{G_T(x_0)}^{\zeta_T(t)}g_0(u)\,du -2\int
_{0}^{t}g_0\bigl({\zeta}_{T}(s)\bigr)\,d\eta_T(s)+\gamma
^{(1)}_T(t)-\gamma^{(2)}_T(t),
\end{align*}
%
where
\begin{align*}
\gamma^{(1)}_T(t)&=2\int_{x_0}^{\xi_T(t)}\hat{q}_T(u)\,du,\qquad
\gamma^{(2)}_T(t)=2\int_{0}^{t}\hat{q}_T\bigl({\xi}_{T}(s)
\bigr)\,dW_T(s),\\
\hat{q}_T(x)&=\left(f'_T(x)\int_{0}^{x}\frac{g_T(v)}{f'_T(v)}\,
dv-g_0\bigl(G_T(x)\bigr)G_T^{\prime}(x)\right).
\end{align*}

Denote $P_{NT}=\pr\{\mathop{\sup}_{0\leq t\leq L}
|\xi_{T}(t)|>N\}$. It is clear that, for any constants
$\varepsilon>0$, $N>0$, and $L>0$, we have the inequalities
\begin{align*}
\pr\Big\{\mathop{\sup}\limits_{0\leq t\leq L} \big|\gamma
^{(1)}_T(t)\big|>\varepsilon\Big\}&\leq P_{NT}+\frac{2}{\varepsilon
}\M\mathop{\sup}\limits_{0\leq t\leq L}\left|\int_{x_0}^{\xi
_T(t)}\hat{q}_T(u)\,du\right| \chi_{\{|\xi_T(t)|\leq N\}}\\
&\leq P_{NT}+\frac{2}{\varepsilon}\int_{-N}^{N}\left|\hat
{q}_T(u)\right|\,du \leq P_{NT}+\frac{4}{\varepsilon}N \mathop{\sup
}\limits_{|x|\leq N}\left|\hat{q}_T(x)\right|
\end{align*}
and
\begin{align*}
&\pr\Big\{\mathop{\sup}\limits_{0\leq t\leq L} \big|\gamma
^{(2)}_T(t)\big|>\varepsilon\Big\}\\
&\quad\leq P_{NT}+\frac{4}{\varepsilon
^2}\M\mathop{\sup}\limits_{0\leq t\leq L} \left|\int
_{0}^{t}\hat{q}_T\bigl({\xi}_{T}(s)\bigr)\chi_{\{|\xi_T(s))\leq N|\}}\,dW_T(s)
\right|^2\\
&\quad\leq P_{NT}+\frac{16}{\varepsilon^2}\M\int_{0}^{L}\hat
{q}^2_T\bigl({\xi}_{T}(s)\bigr)\chi_{\{|\xi_T(s))\leq N|\}}\,ds \leq
P_{NT}+\frac{16}{\varepsilon^2}L\mathop{\sup}\limits_{|x|\leq N}\big
|\hat{q}^2_T(x)\big|.
\end{align*}

The inequality $|G_T(x)|\geq C |x|^{\alpha}$, $\alpha>0$, together
with convergence (\ref{o5}), implies that
\begin{equation}\label{o19}
\mathop{\lim}\limits_{N \to\infty}\overline{\mathop{\lim}\limits
_{T \to\infty}}P_{NT}=0.
\end{equation}

Thus, using the conditions of Theorem \ref{th4}, we get the convergence
%
\[
\sup_{0\leq t\leq L}\big|\gamma^{(k)}_T(t)\big|\stackrel{\pr}{\to}
0,\quad k=1,2,
\]
%
as $T\to\infty$.

The same arguments as we used establishing (\ref{o11}) yield that
\begin{equation}\label{o21}
\tilde{\beta}_{T_n}^{(1)} (t)=2\int_{G_{T_n}(x_0)}^{\tilde{\zeta
}_{T_n}(t)}g_0(u)\,du -2\int_{0}^{t}g_0\big(\tilde{\zeta
}_{T_n}(s)\big)\,d\tilde{\eta}_{T_n}(s)+\tilde{\gamma
}^{(1)}_{T_n}(t)-\tilde{\gamma}^{(2)}_{T_n}(t),
\end{equation}
where
\begin{align*}
&\sup_{0\leq t\leq L}\big|\tilde{\zeta}_{T_n}(t)-\tilde{\zeta}(t)\big
|\stackrel{\tilde{\pr}}{\to} 0, \qquad\sup_{0\leq t\leq L}\left|\tilde
{\eta}_{T_n}(t)-\tilde{\eta}(t)\right|\stackrel{\tilde{\pr}}{\to} 0,\\
&\sup_{0\leq t\leq L}\big|\tilde{\gamma}^{(k)}_{T_n}(t)\big|\stackrel
{\tilde{\pr}}{\to} 0,\quad k=1,2,
\end{align*}
as $T_n\to\infty$ for all $L>0$. According to (\ref{o13}) with $\tilde
{\eta}(t)$ defined in (\ref{o14}), the process $(\tilde{\zeta}
(t),\hat{W}(t))$ satisfies Eq.~(\ref{o3}).

By Lemma \ref{lm5} the finite-dimensional distributions of the
stochastic process $\beta_{T_n}^{(1)} (t)$ coincide with those of the
process $\tilde{\beta}_{T_n}^{(1)} (t)$.
Using Lemma \ref{lm3}, we can pass to the limit as $T_n\to\infty$ in
(\ref{o21}) and obtain
\begin{equation}\label{o22}
\sup_{0\leq t\leq L}\big|\tilde{\beta}_{T_n}^{(1)}(t)-\tilde{\beta
}^{(1)}(t)\big|\stackrel{\tilde{\pr}}{\to} 0
\end{equation}
as $T_n\to\infty$, where
\begin{align*}
\tilde{\beta}^{(1)}(t)&=2\int_{y_0}^{\tilde{\zeta} (t)}g_0(u)\,du
-2\int_{0}^{t}g_0\bigl(\tilde{\zeta}(s)\bigr)\,d\tilde{\eta}(s)\\
&=2\int_{y_0}^{\tilde{\zeta} (t)}g_0(u)\,du -2\int
_{0}^{t}g_0\bigl(\tilde{\zeta}(s)\bigr)\,d\tilde{\zeta}(s)+2\int
_{0}^{t}g_0\bigl(\tilde{\zeta}(s)\bigr)
a_0\bigl(\tilde{\zeta}(s)\bigr)\,ds,
\end{align*}
and $\tilde{\zeta}(t)$ is a solution of Eq.~(\ref{o3}). Therefore, we
have that Theorem \ref{th4} holds for the process ${\beta}_{T_n}^{(1)} (t)$
as $T_n\to\infty$.
Since the subsequence $T_n\to\infty$ is arbitrary and since a solution
$\zeta(t)$ of Eq.~(\ref{o3}) is weakly unique, the proof of Theorem \ref
{th4} is complete.

Proof of Theorem \ref{th5}.
It is clear that
\begin{equation}\label{o23}
\beta_T^{(2)} (t)=\int_{0}^{t} g_0\bigl(\zeta_T(s)\bigr)\,d\eta_T(s)+\gamma
_T (t),
\end{equation}
where
\[
\gamma_T (t)=\int_{0}^{t}{q}_T\bigl({\xi}_{T}(s)\bigr)\,
dW_T(s),\qquad q_T(x)=g_T(x)-g_0\bigl(G_T(x)\bigr)G_T^{\prime}(x).
\]

The process $\gamma_T (t)$ is a continuous martingale with the
quadratic characteristics
\[
\langle{\gamma_T}\rangle(t)=\int_{0}^{t}q_T^2\bigl(\xi_T(s)\bigr)\,ds
\quad\text{for all} \ T>0.\vadjust{\eject}
\]

According to the conditions of Theorem \ref{th5}, the functions
$q_T^2(x)$ satisfy the conditions of Lemma \ref{lm2}. Thus, for any
$L>0$, we have the convergence $\langle{\gamma_T}\rangle(L)\stackrel
{{\pr}}{\to} 0$ as $T\to\infty$.

The following inequality holds for any constants $\varepsilon>0$ and
$\delta>0$:
\[
\pr\Bigl\{\mathop{\sup}\limits_{0\leq t\leq L} \big|\gamma_T(t)
\big|>\varepsilon\Bigr\}\leq\delta+\pr\bigl\{ \langle{\gamma_T}\rangle
(L)>\varepsilon^2\delta\bigr\}
\]
(see \cite{book2}, \S3, Theorem 2), which implies the relation
\begin{equation}\label{o24}
\sup_{0\leq t\leq L}\big|\gamma_T(t)\big|\stackrel{\pr}{\to} 0
\end{equation}
as $T\to\infty$.

Then, similarly to representation (\ref{o11}), on a certain probability
space $(\tilde{\varOmega},\tilde{\Im} , \tilde{\pr})$, for an arbitrary
subsequence $T_n$, we get the equality
\[
\tilde{\beta}_{T_n}^{(2)} (t)=\int_{0}^{t} g_0\bigl(\tilde{\zeta
}_{T_n}(s)\bigr)\,d\tilde{\eta}_{T_n}(s)+\tilde{\gamma}_{T_n} (t),
\]
where
\[
\sup_{0\leq t\leq L}\big|\tilde{\zeta}_{T_n}(t)-\tilde{\zeta}(t)\big
|\stackrel{\tilde{\pr}}{\to} 0, \qquad\sup_{0\leq t\leq L}\big|\tilde
{\eta}_{T_n}(t)-\tilde{\eta}(t)\big|\stackrel{\tilde{\pr}}{\to} 0,\qquad
\sup_{0\leq t\leq L}\big|\tilde{\gamma}_{T_n}(t)\big|\stackrel{\tilde
{\pr}}{\to} 0
\]
as $T_n\to\infty$ for any $L>0$, where the process $(\tilde{\zeta}
(t),\hat{W}(t))$ satisfies Eq.~(\ref{o3}), $\tilde{\eta}(t)$ is
defined in (\ref{o14}),
and the processes $\tilde{\beta}_{T_n}^{(2)} (t)$ and ${\beta
}_{T_n}^{(2)} (t)$ are stochastically equivalent.

Similarly to the proof of convergence (\ref{o22}), we obtain
\[
\sup_{0\leq t\leq L}\big|\tilde{\beta}_{T_n}^{(2)}(t)-\tilde{\beta
}^{(2)}(t)\big|\stackrel{\tilde{\pr}}{\to} 0
\]
as $T_n\to\infty$, where
%
\begin{align*}
\tilde{\beta}^{(2)}(t)&=\int_{0}^{t}g_0\bigl(\tilde{\zeta
}(s)\bigr)\,d\tilde{\eta}(s)\\
&=\int_{0}^{t}g_0\bigl(\tilde{\zeta}(s)\bigr)\,d\tilde{\zeta
}(s)-\int_{0}^{t}g_0\bigl(\tilde{\zeta}(s)\bigr)
a_0\bigl(\tilde{\zeta}(s)\bigr)\,ds.
\end{align*}
%

Thus, the process $\tilde{\beta}_{T_n}^{(2)} (t)$ weakly converges, as
$T_n\to\infty$, to the process~$\tilde{\beta}^{(2)} (t)$. Since the
subsequence $T_n\to\infty$ is arbitrary and since the processes $\tilde
{\beta}_{T_n}^{(2)} (t)$ and ${\beta}_{T_n}^{(2)} (t)$ are
stochastically equivalent, the proof of Theorem \ref{th5} is complete.

Proof of Theorem \ref{th6}.
It is clear that
\[
I_T(t)=F_0\bigl(\zeta_T(t)\bigr)+\int_{0}^{t} g_0\bigl(\zeta_T(s)\bigr)\,d\eta
_T(s)+\alpha_T (t)+\gamma_T (t),
\]
where
\begin{align*}
\alpha_T (t)&=F_T\bigl({\xi}_{T}(t)\bigr)-F_0\bigl({\zeta}_{T}(t)
\bigr),\qquad\eta_T(t)=\int_{0}^{t}G_T^{\prime}\bigl(\xi_T(s)\bigr)\,dW_T(s),\\
\gamma_T (t)&=\int_{0}^{t}{q}_T\bigl({\xi}_{T}(s)\bigr)\,
dW_T(s),\qquad q_T(x)=g_T(x)-g_0\bigl(G_T(x)\bigr)G_T^{\prime}(x).
\end{align*}

Denote, as before, $P_{NT}=\pr\{\mathop{\sup}_{0\leq
t\leq L} |\xi_{T}(t)|>N\}$. Since for any constants
$\varepsilon>0$, $N>0$, and $L>0$, we have the inequality
\begin{align*}
&\pr\Big\{\mathop{\sup}\limits_{0\leq t\leq L} \big|F_T\bigl({\xi
}_{T}(t)\bigr)-F_0\bigl(G_T\bigl(\xi_T(t)\bigr)\bigr)\big|>\varepsilon\Big
\}\\
&\quad\leq P_{NT}+\frac{2}{\varepsilon}\M\mathop{\sup}\limits_{0\leq t\leq
L}\big|F_T\bigl({\xi}_{T}(t)\bigr)-F_0\bigl(G_T\bigl(\xi_T(t)\bigr)\bigr)
\big| \chi_{\{|\xi_T(t)|\leq N\}}\\
&\quad\leq P_{NT}+\frac{2}{\varepsilon} \mathop{\sup}\limits_{|x|\leq
N}\big|F_T(x)-F_0\big(G_T\xch{(x)\big)}{(x)}\big|,
\end{align*}
we can apply conditions of Theorem
\ref{th6} and convergence (\ref{o19}) to get that
%
\[
\sup_{0\leq t\leq L}\big| \alpha_T(t)\big|\stackrel{\pr}{\to} 0
\]
%
as $T\to\infty$. The proof of the fact that, for $\gamma_T (t)$, an
analogue of convergence (\ref{o24}) holds is literally the same as in
the proof of Theorem \ref{th5}. Then, we can apply Skorokhod's
convergent subsequence principle to the process $\left(\zeta_T(t), \eta
_T(t),\alpha_T(t), \gamma_T (t)\right)$
and, similarly to representation (\ref{o11}), obtain the following
equality for an arbitrary subsequence $T_n$ in a certain probability
space $(\tilde{\varOmega},\tilde{\Im} , \tilde{\pr})$:
%
%
\[
\tilde{I}_{T_n}(t)=F_0\bigl(\tilde{\zeta}_{T_n}(t)\bigr)+\int_{0}^{t}
g_0\bigl(\tilde{\zeta}_{T_n}(s)\bigr)\,d\tilde{\eta}_{T_n}(s)+\tilde{\alpha
}_{T_n} (t)+\tilde{\gamma}_{T_n} (t),
\]
%
where, as $T_n\to\infty$, for any $L>0$,
\begin{align*}
&\sup_{0\leq t\leq L}\big|\tilde{\zeta}_{T_n}(t)-\tilde{\zeta}(t)
\big|\stackrel{\tilde{\pr}}{\to} 0, \qquad\sup_{0\leq t\leq L}\big|\tilde
{\eta}_{T_n}(t)-\tilde{\eta}(t)\big|\stackrel{\tilde{\pr}}{\to} 0,\\
&\sup_{0\leq t\leq L}\big|\tilde{\alpha}_{T_n}(t)\big|\stackrel{\tilde
{\pr}}{\to} 0,\qquad\sup_{0\leq t\leq L}\big|\tilde{\gamma
}_{T_n}(t)\big|\stackrel{\tilde{\pr}}{\to} 0.
\end{align*}

To complete the proof of Theorem \ref{th6}, we repeat the same
arguments as in the proof of Theorem \ref{th5}.

Proof of Theorem \ref{th7}.
According to the It\^{o} formula, the process $\zeta_T(t)\,{=}\,f_T(\xi
_T(t))$ satisfies the equation $d\zeta_T (t)=\hat{\sigma}_T(\zeta
_T(t))\,dW_T(t)$, where $\hat{\sigma}_T(x)=f_T'\left(\varphi_T (x)\right
)$, $\varphi_T (x)$ is the inverse function of the function $f_T (x)$,
and $\zeta_T(0)=f_T(x_0)\to y_0$ as $T\to\infty$. In addition, the
following equality holds:
\[
I_T(t)=\hat{F}_T\bigl(\zeta_T(t)\bigr)+\int_{0}^{t} \hat{g}_T\bigl(\zeta_T(s)\bigr)\,
d\zeta_T(s),
\]
where $\hat{F}_T(x)=F_T(\varphi_T (x))$ and $\hat{g}_T(x)=g_T(\varphi_T
(x))\cdot\hat{\sigma}_T^{-1}(x)$.

It is easy to see that condition 1 of the present theorem implies that
\[
\int_{0}^{x}\frac{dv}{\hat{\sigma}_T^{2}(v)}\to\int
_{0}^{x}\frac{dv}{{\sigma}_0^{2}(v)}
\]
as $T\to\infty$ for all $x$, whereas condition 2 implies that
\[
\mathop{\sup}\limits_{|x|\leq N} \big|\hat
{F}_T(x)+c_T^{(2)}+c_T^{(1)}x-F_0(x)\big|\to0
\]
and
\[
\int_{-N}^{N} \big|\hat{g}_T(x)-c_T^{(1)}-g_0(x)\big|^2\,dx
\to0
\]
as $T\to\infty$ for any $N>0$.

This means that the necessary and sufficient conditions of weak
convergence of the process $\left(\zeta_T (t),I_T(t)\right)$ as $T\to
\infty$ to the process $\left(\zeta(t),I(t)\right)$ from \cite
{paper11} hold with
$b_T=c_T^{(1)}$ and $a_T=c_T^{(2)}$.

\section{Auxiliary results}\label{section4}

\begin{lem}\label{lm1}
Let $\xi_T$ be a solution of Eq.~\eqref{o1} from the class $K\left(G_T
\right)$. Then, for any $N>0$ and any Borel set $B\subset\left[-N;
N\right]$, there exists a constant $C_L$ such that
\[
\int_{0}^{L}\pr\bigl\{G_T\bigl(\xi_T(s)\bigr)\in B\bigr\}\,ds\leq C_L \psi
\bigl(\lambda(B)\bigr),
\]
where $\lambda(B)$ is the Lebesgue measure of $B$, and $\psi\left
(|x|\right)$ is a bounded function satisfying $\psi\left(|x|\right)\to
0$ as $|x|\to0$.
\end{lem}

\begin{proof}
Consider the function
\[
\varPhi_T(x)=2\int_{0}^{x}f'_T(u)\Biggl(\int_{0}^{u}\frac{\chi
_B\bigl(G_T(v)\bigr)}{f'_T(v)}\,dv\Biggr)\,du.
\]

The function $\varPhi_T(x)$ is continuous, the derivative $\varPhi'_T(x)$ of
this function
is continuous, and the second derivative $\varPhi''_T(x)$ exists a.e.\
with respect to the Lebesgue measure and is locally bounded. Therefore,
we can apply the It\^{o} formula to the process $\varPhi_T(\xi_T(t))$,
where $\xi_T(t)$ is a solution of Eq.~(\ref{o1}).

Furthermore,
\[
\varPhi'_T(x)a_T(x)+\frac{1}{2}\,\varPhi''_T(x)=\chi_B(x)
\]
a.e.\ with respect to the Lebesgue measure. Using the latter equality,
we conclude that
\begin{equation}\label{l1}
\int_{0}^{t}\chi_B\bigl(\zeta_T(s)\bigr)\,ds=\varPhi_T\bigl(\xi
_T(t)\bigr)-\varPhi_T(x_0)- \int_{0}^{t}\varPhi'_T\bigl(\xi_T(s)\bigr)\,dW_T(s)
\end{equation}
with probability one for all $t\geq0$, where $\zeta_T (t)=G_T(\xi_T(t))$.
Hence, using the properties of stochastic integrals, we obtain that
\begin{equation}\label{l2}
\int_{0}^{t}\pr\bigl\{\zeta_T (s)\in B\bigr\}\,ds=\M\bigl[\varPhi_T\bigl(\xi
_T(t)\bigr)-\varPhi_T(x_0)\bigr].
\end{equation}

According to condition $(A_2)$ and inequality $|G_T(x)|\geq C
|x|^{\alpha}$, $C>0$, $\alpha>0$, we have
\[
\big|\varPhi_T(x)-\varPhi_T(x_0)\big|\leq2\psi\bigl(\lambda(B)
\bigr)\left[1+|x|^m\right]\leq2\psi\bigl(\lambda(B)\bigr)
\bigl[1+C^{-\frac{m}{\alpha}}\big|G_T(x)\big|^{\frac{m}{\alpha}}\bigr].
\]

Hence, using inequality (\ref{o6}), we obtain that\vadjust{\eject}
\[
\big|\M\bigl[\varPhi_T\bigl(\xi_T(L)\bigr)-\varPhi_T(x_0)\bigr]\big|\leq C_L \psi
\bigl(\lambda(B)\bigr)
\]
for some constant $C_L$. The latter inequality and Eq.~(\ref{l2}) prove
Lemma \ref{lm1}.
\end{proof}

\begin{lem}\label{lm2} Let $\xi_T$ be a solution of Eq.~\eqref{o1} from
the class $K\left(G_T \right)$. If, for measurable locally bounded
functions $q_T(x)$,
condition $(A_3)$ holds, then, for any $L>0$,
\[
\sup_{0\leq t\leq L}\Biggl|\int_{0}^{t} q_T\bigl(\xi_T(s)\bigr)\,ds
\Biggr|\stackrel{\pr}{\to} 0
\]
as $T\to\infty$.
\end{lem}

\begin{proof}
Consider the function
\[
\varPhi_T(x)=2\int_{0}^{x}f'_T(u)\Biggl(\int_{0}^{u}\frac
{q_T(v)}{f'_T(v)}\,dv\Biggr)\,du.
\]
The same arguments as used to obtain Eq.~(\ref{l1}) yield that
\begin{equation}\label{l3}
\int_{0}^{t}q_T\bigl(\xi_T(s)\bigr)\,ds=\varPhi_T\bigl(\xi_T(t)\bigr)-\varPhi_T(x_0)- \int
_{0}^{t}\varPhi'_T\bigl(\xi_T(s)\bigr)\,dW_T(s).
\end{equation}

It is clear that, for any constants $\varepsilon>0$, $N>0$, and $L>0$,
we have the inequalities
\begin{align*}
&\pr\Big\{\mathop{\sup}\limits_{0\leq t\leq L} \big|\varPhi_T\bigl(\xi
_T(t)\bigr)\big|>\varepsilon\Big\}\leq P_{NT}+\frac{2}{\varepsilon}\int
_{-N}^{N}f'_T(u)\left|\int_{0}^{u}\frac{q_T(v)}{f'_T(v)}\,
dv\right|\,du,\\
&\pr\Biggl\{\mathop{\sup}\limits_{0\leq t\leq L} \left|\int
_{0}^{t}\varPhi'_T\bigl(\xi_T(s)\bigr)\,dW_T(s)\right|>\varepsilon\Biggr\}\\
&\quad\leq P_{NT}+\frac{1}{\varepsilon^2}\M\mathop{\sup}\limits_{0\leq
t\leq L} \left|\int_{0}^{t}\varPhi'_T\bigl(\xi_T(s)\bigr)\chi_{\{|\xi
_T(s))\leq N|\}}\,dW_T(s)\right|^2\\
&\quad\leq P_{NT}+\frac{4}{\varepsilon^2}\M\int_{0}^{L}\left[\varPhi
'_T\bigl(\xi_T(s)\bigr)\right]^2\chi_{\{|\xi_T(s))\leq N|\}}\,ds\\
&\quad\leq P_{NT}+\frac{16}{\varepsilon^2}L\mathop{\sup}\limits_{|x|\leq
N}\left[f'_T(x)\left|\int_{0}^{x}\frac{q_T(v)}{f'_T(v)}\,dv\right
|\right]^2,
\end{align*}
where $P_{NT}=\pr\{\mathop{\sup}_{0\leq t\leq L} |\xi
_{T}(t)|>N\}$.

Therefore, using convergence (\ref{o19}), we obtain
\[
\sup_{0\leq t\leq L}\big|\varPhi_T\bigl(\xi_T(t)\bigr)-\varPhi_T(x_0)\big|\stackrel
{\pr}{\to} 0
\]
and
\[
\sup_{0\leq t\leq L}\Biggl|\int_{0}^{t}\varPhi'_T\bigl(\xi_T(s)\bigr)\,
dW_T(s)\Biggr|\stackrel{\pr}{\to} 0
\]
as $T\to\infty$. Thus, Eq.~(\ref{l3}) implies the statement of Lemma
\ref{lm2}.
\end{proof}

\begin{lem}\label{lm3}
Let $\xi_T$ be a solution of Eq.~\eqref{o1} from the class $K\left(G_T
\right)$, and let $\zeta_T (t)=G_T(\xi_T(t))\stackrel{\pr}{\to}\zeta
(t)$ as $T\to\infty$. Then for any measurable locally bounded function
$g(x)$, we have the convergence\vadjust{\eject}
\[
\sup_{0\leq t\leq L}\Biggl|\int_{0}^{t} g\bigl(\zeta_T(s)\bigr)\,ds-\int
_{0}^{t} g\bigl(\zeta(s)\bigr)\,ds\Biggr|\stackrel{\pr}{\to} 0
\]
as $T\to\infty$ for any constant $L>0$.
\end{lem}

\begin{proof}
Let $\varphi_N(x)=1$ for $|x|\leq N$, $\varphi_N(x)=N+1-|x|$ for
$|x|\in\left[N,N+1\right]$, and $\varphi_N(x)=0$ for $|x|> N+1$.
Then, for all $T>0$ and $L>0$,
\begin{align*}
&\pr\Biggl\{\mathop{\sup}\limits_{0\leq t\leq L} \left|\int
_{0}^{t}\bigl[g\bigl(\zeta_T(s)\bigr)-g\bigl(\zeta_T(s)\bigr)\varphi_N\bigl(\zeta_T(s)\bigr)\bigr]\,
ds\right|{>}\,0\Biggr\}\leq\pr\Big\{\mathop{\sup}\limits_{0\leq t\leq
L} \big|\zeta_{T}(t)\big|\,{>}\,N\Big\},\\
&\pr\Biggl\{\mathop{\sup}\limits_{0\leq t\leq L} \left|\int
_{0}^{t}\bigl[g\bigl(\zeta(s)\bigr)-g\bigl(\zeta(s)\bigr)\varphi_N\bigl(\zeta(s)\bigr)\bigr]\,
ds\right|>0\Biggr\}\\
&\quad\leq\pr\Big\{\mathop{\sup}\limits_{0\leq t\leq
L} \big|\zeta(t)\big|>N\Big\}\leq\overline{\mathop{\lim}\limits_{T \to\infty}}\pr\Big\{\mathop
{\sup}\limits_{0\leq t\leq L} \big|\zeta_{T}(t)\big|>N\Big\}.
\end{align*}

According to Theorem \ref{th2}, convergence (\ref{o5}) holds for the process
$\zeta_{T}(t)$. So, to complete the proof of Lemma \ref{lm3}, we need
to establish that
\begin{equation}\label{l4}
\int_{0}^{L}\big|g\bigl(\zeta_T(s)\bigr)\varphi_N\bigl(\zeta_T(s)\bigr)-g\bigl(\zeta
(s)\bigr)\varphi_N\bigl(\zeta(s)\bigr)\big|\,ds\stackrel{\pr}{\to} 0
\end{equation}
as $T\to\infty$.

First, assume that the function $g(x)$ is continuous. Then
\[
g\bigl(\zeta_T(s)\bigr)\varphi_N\bigl(\zeta_T(s)\bigr)-g\bigl(\zeta(s)\bigr)\varphi_N\bigl(\zeta
(s)\bigr)\stackrel{\pr}{\to} 0
\]
as $T\to\infty$ for all $0\leq s\leq L$, and $\left|g(x)\varphi
_N(x)\right|\leq C_N$ for all $x$. Thus, by Lebesgue's dominated
convergence theorem we have convergence (\ref{l4}). Second, let the
function $g(x)$ be measurable and locally bounded. Then, using Luzin's
theorem, we conclude that, for any $\delta>0$, there exists a
continuous function $g^{\delta}(x)$ that coincides with $g(x)$ for
$x\notin B^{\delta}$, where $B^{\delta}\subset\left[-N-1,N+1\right
]$, and the Lebesgue measure satisfies the inequality $\lambda
(B^{\delta})< \delta$. Thus, for every $\delta>0$, convergence
(\ref{l4}) holds for the function $g^{\delta}(x)$. Since, for any
$\varepsilon>0$,
\begin{align*}
&\pr\left\{\int_{0}^{L}\left|g\bigl(\zeta_T(s)\bigr)\varphi_N\bigl(\zeta
_T(s)\bigr)-g^{\delta}\bigl(\zeta_T(s)\bigr)\varphi_N\bigl(\zeta_T(s)\bigr)\right|\,
ds>\varepsilon\right\}\\
&\quad\leq\frac{1}{\varepsilon}\M\int_{0}^{L}\left|g\bigl(\zeta
_T(s)\bigr)\varphi_N\bigl(\zeta_T(s)\bigr)-g^{\delta}\bigl(\zeta_T(s)\bigr)\varphi_N\bigl(\zeta
_T(s)\bigr)\right|\chi_{\{B^{\delta}\}}\bigl(\zeta_T(s)\bigr)\,ds\\
&\quad\leq\frac{C_N}{\varepsilon}\int_{0}^{L}\pr\left\{\zeta_T(s)\in
B^{\delta}\right\}\,ds,\\
&\pr\left\{\int_{0}^{L}\left|g\bigl(\zeta(s)\bigr)\varphi_N\bigl(\zeta
(s)\bigr)-g^{\delta}\bigl(\zeta(s)\bigr)\varphi_N\bigl(\zeta(s)\bigr)\right|\,ds>\varepsilon
\right\}\\
&\quad\leq\frac{C_N}{\varepsilon}\int_{0}^{L}\pr\left\{\zeta(s)\in
B^{\delta}\right\}\,ds\leq\frac{C_N}{\varepsilon}\;\overline{\mathop
{\lim}\limits_{T \to\infty}}\;\int_{0}^{L}\pr\left\{\zeta
_T(s)\in B^{\delta}\right\}\,ds,
\end{align*}
taking into account Lemma \ref{lm1}, we conclude that convergence (\ref
{l4}) holds for such a function $g(x)$ as well.\vadjust{\eject}
\end{proof}

\begin{lem}\label{lm4}
Let $\xi_T$ be a solution of Eq.~\eqref{o1} from the class $K\left(G_T
\right)$, and let $\zeta_T (t)=G_T(\xi_T(t))\stackrel{\pr}{\to}\zeta
(t)$ and $\eta_T(t)=\int_{0}^{t}G_T^{\prime}(\xi_T(s))\,
dW_T(s)\stackrel{\pr}{\to}\eta(t)$ as $T\to\infty$. Then, for
measurable locally bounded functions $g(x)$, we have the convergence
\[
\sup_{0\leq t\leq L}\left|\int_{0}^{t} g\bigl(\zeta_T(s)\bigr)\,d\eta
_T(s)-\int_{0}^{t} g\bigl(\zeta(s)\bigr)\,d\eta(s)\right|\stackrel{\pr}{\to
} 0
\]
as $T\to\infty$ for any constant $L>0$.
\end{lem}

\begin{proof}
Similarly to the proof of Lemma \ref{lm3}, it suffices to obtain an
analogue of convergence
(\ref{l4}), that is, to get that, for any $N>0$ and $L>0$,
\begin{equation}\label{l5}
\sup_{0\leq t\leq L}\left|\int_{0}^{t}g\bigl(\zeta_T(s)\bigr)\varphi
_N\bigl(\zeta_T(s)\bigr)\,d\eta_T(s)-\int_{0}^{t}g\bigl(\zeta(s)\bigr)\varphi_N\bigl(\zeta
(s)\bigr)\,d\eta(s)\right|\stackrel{\pr}{\to} 0
\end{equation}
as $T\to\infty$, where $\varphi_N(x)$ is defined in the proof of Lemma
\ref{lm3}. The proof of convergence (\ref{l5}) for a continuous
function $g(x)$ is similar to that of the corresponding theorem in \cite
{book6}, Chapter~2, \S6. The explicit form of the quadratic
characteristic $\langle\eta_T\rangle(t)$ of the martingale $\eta_T
(t)$ and condition $(A_1)$ imply the inequality
\[
\int_{0}^{L}\bigl[\varphi_N\bigl(\zeta_T(t)\bigr)\bigr]^2\,d\langle\eta
_T\rangle(t)\leq C_N L,
\]
which is used for the proof of convergence (\ref{l5}). The extension of
such a convergence to the class of measurable locally bounded functions
is based on Lemma~\ref{lm1} and is provided similarly to the proof of
Lemma \ref{lm3}.
\end{proof}

\begin{lem}\label{lm5}
Let $\xi_T$ be a solution of Eq.~\eqref{o1} belonging to the class
$K\left(G_T \right)$, and let the stochastic process $\left(\zeta_T
(t),\eta_T (t)\right)$, with $\zeta_T (t)=G_T(\xi_T(t)) $ and $\eta
_T(t)=\int_{0}^{t}G_T^{\prime}(\xi_T(s))\,dW_T(s)$ be
stochastically equivalent to the process $(\tilde{\zeta}_T (t),
\tilde{\eta}_T (t))$. Then the process
\[
\int_{0}^{t}g\bigl(\zeta_T(s)\bigr)\,ds+\int_{0}^{t}q\bigl(\zeta_T(s)\bigr)\,
d\eta_T(s),
\]
where $g(x)$ and $q(x)$ are measurable locally bounded functions, is
stochastically equivalent to the process
\[
\int_{0}^{t}g\bigl(\tilde{\zeta}_T(s)\bigr)\,ds+\int_{0}^{t}q\bigl(\tilde
{\zeta}_T(s)\bigr)\,d\tilde{\eta}_T(s).
\]
\end{lem}

\begin{proof}
The proof is the same as that \xch{of}{of the} Theorem 2 from \cite{paper9}.
\end{proof}

\end{document}